\documentclass[10pt,arxiv,scale=0.81]{agn_article}

\usepackage{agn_all}
\usepackage{framed}
\usepackage{tabularx}
\usepackage{caption}
\usepackage{lscape}
\usepackage{rotating}
\usepackage{float}
\usepackage{graphicx}
\usepackage{wrapfig}

\newcommand{\citet}[2][]{\citeauthor{#2} \cite[#1]{#2}}

\newcommand{\quoteref}[1]{\enquote{\textit{#1}}}
\newcommand{\quoteesc}[1]{\textnormal{[#1]}}

\begin{document}
\title{\vspace{-2cm}Again anti-plane shear}
\date{\today}
\author{%
	Jendrik Voss\thanks{%
		Corresponding author: Jendrik Voss,\quad Lehrstuhl f\"{u}r Nichtlineare Analysis und Modellierung, Fakult\"{a}t f\"{u}r Mathematik, Universit\"{a}t Duisburg-Essen, Thea-Leymann Str. 9, 45127 Essen, Germany; email: max.voss@stud.uni-due.de%
	}\quad and \quad%
	Herbert Baaser\thanks{%
		Herbert Baaser,\quad Mechanical Engineering,\quad University of Applied Sciences Bingen, Germany, email: h.baaser@th-bingen.de}
	\quad and\quad%
	Robert J.\ Martin\thanks{%
		Robert J.\ Martin,\quad Lehrstuhl f\"{u}r Nichtlineare Analysis und Modellierung, Fakult\"{a}t f\"{u}r Mathematik, Universit\"{a}t Duisburg-Essen, Thea-Leymann Str. 9, 45127 Essen, Germany; email: robert.martin@uni-due.de%
	}\quad and\quad
	Patrizio Neff\thanks{%
		Patrizio Neff,\quad Head of Lehrstuhl f\"{u}r Nichtlineare Analysis und Modellierung, Fakult\"{a}t f\"{u}r	Mathematik, Universit\"{a}t Duisburg-Essen, Thea-Leymann Str. 9, 45127 Essen, Germany, email: patrizio.neff@uni-due.de%
		}%
}
\maketitle
\begin{center}
	Dedicated to Krzysztof Che{\l}mi{\'n}ski on the occasion of his 60th birthday.
\end{center}
\begin{abstract}
\noindent
	We reconsider anti-plane shear deformations of the form $\varphi(x)=(x_1,\,x_2,\,x_3+u(x_1,x_2))$ based on prior work of Knowles and relate the existence of anti-plane shear deformations to fundamental constitutive concepts of elasticity theory like polyconvexity, rank-one convexity and tension-compression symmetry. In addition, we provide finite-element simulations to visualize our theoretical findings.
\end{abstract}
\\[.65em]

\section{Introduction}
The anti-plane shear problem is considered one of the classical challenges in applied nonlinear elasticity theory \cite{knowles1976finite,knowles1977note,knowles1977finite}. The essence of this problem is to consider only a very special and simple deformation mode, the so-called anti-plane shear (\textbf{APS}), which allows for a reduction of the governing set of equations to an analytically more tractable form, in the compressible as well as the incompressible case. It has been traditional (although not mandatory) to interpret the anti-plane shear problem in a certain nontrivial sense: namely which nonlinear elastic formulations (with nonlinear energies) allow for solutions in APS-form provided only the boundary data is in APS-form \cite{horgan1994antiplane,horgan1995anti,Pucci2013,pucci2013note,saccomandi2016dy}.\par
In contrast to this established approach, Gao \cite{gao2015remarks,gao2017remarks,gao2015duality,gao2016analytical} has recently re-interpreted the APS-problem as a simple search for minimizers of the energy functional within the restricted class of APS-deformations. Obviously, the two approaches share some concepts but are, in general, distinct from each other. In this paper, we clarify the differences between both approaches, including numerical examples to highlight the unsuitability of the latter approach in the general case. For simplicity, we pose the APS-problem only for pure Dirichlet boundary data and and restrict our considerations to the isotropic case.\par
We also give a counterexample to a recent statement from the 2015 Int.\ J.\ Eng.\ Sci.\ article \enquote{\emph{On the determination of semi-inverse solutions of nonlinear Cauchy elasticity: The not so simple case of anti-plane shear}} by Pucci, Rajagopal and Saccomandi \cite{pucci2015determination}, erroneously connecting ellipticity and the so-called empirical inequalities.
%
%
%
%
\newpage
\section{Anti-plane shear deformations}
We employ the usual notion of an anti-plane shear deformation.
\begin{definition}
	An \emph{anti-plane shear deformation} (or \textbf{APS}-deformation) is a mapping $\varphi\col\Omega\subset\mathbb{R}^3\to\mathbb{R}^3$ of the form $\varphi(x_1,x_2, x_3)= \left(x_1,\;x_2,\;x_3+u(x_1, x_2)\right)$ with an arbitrary scalar-valued function $u\col\Omega_{xy}\subset\mathbb{R}^2\,\to\,\mathbb{R}$.
\end{definition}
Due to its form, an APS-deformation of a cylinder-shaped body can be identified entirely by the displacement of the bottom or top with the scalar-valued height-function $u(x_1, x_2)$. Let $\alpha\colonequals u_{,x_1}\,,\,\beta\colonequals u_{,x_2}\,,\,\gamma\colonequals\norm{\nabla u}$ and $\gamma^2=\alpha^2+\beta^2\,.$ Then the deformation-gradient $F=\nabla\varphi$ and the left and right Cauchy-Green-deformation tensors $B=FF^T\equalscolon U^2$ and $C=F^TF\equalscolon V^2$ corresponding to an arbitrary APS-deformation are given by
\begin{align}
F=\left(\begin{matrix}
	1 & 0 & 0 \\
	0 & 1 & 0 \\
	\alpha & \beta & 1
\end{matrix}\right),\qquad
B=\left(\begin{matrix}
	1 & 0 & \alpha \\
	0 & 1 & \beta \\
	\alpha & \beta & 1+\gamma^2
\end{matrix}\right),\qquad
C=\left(\begin{matrix}
	1+\alpha^2 & \alpha\beta & \alpha \\
	\alpha\beta & 1+\beta^2 & \beta \\
	\alpha & \beta & 1
\end{matrix}\right).
\label{APSFBC}
\end{align}
In this case, the three isotropic matrix invariants of $B$ (or, equivalently, $C$)
\begin{equation}
I_1=\tr(B)=\norm{F}^2\,,\quad I_2=\frac{1}{2}[(\tr\,B)^2-\,\tr(B^2)]=\tr(\Cof B)=\norm{\Cof F}^2\,,\quad I_3=\det(B)=(\det F)^2 \label{Inv}
\end{equation}
are given by $I_1=I_2=3+\gamma^2$ and $I_3=1$. In particular, APS-deformations always satisfy the condition $\det F=1$ of incompressibility.\par
In this paper, we discuss the deformation $\varphi$ of an elastic isotropic body $\Omega$, cylinder-shaped in its natural state, induced by given boundary conditions. We assume a hyperelastic material behavior, i.e.\ that the resulting deformation $\varphi$ is stationary with respect to the energy functional $I(\varphi)=\int_\Omega W(\nabla\varphi )\,\dx$ for some nonlinear elastic energy potential $W(F)$. The given boundary conditions are assumed to be satisfiable by an APS-deformation of the whole body, which is the case if and only if the Dirichlet boundary conditions only contain consistent $x_3$-shifting.\footnote{Possible boundary conditions are Dirichlet or Neumann boundary conditions which permit an APS-deformation of the surface $\partial\Omega$ of $\Omega$. Here, we restrict our attention to Dirichlet boundary condition for simplicity of exposition.} Under such boundary conditions, we investigate whether or not an APS-deformation of the whole body $\Omega$ exists which is stationary with respect to the energy functional $I(\varphi)$.\par
Of course, due to the involved nonlinearity, the existence of minimizers or solutions to the corresponding Euler-Lagrange equations is not guaranteed without further assumptions (like polyconvexity, cf.\ Section \ref{sec:convex}) on the energy function.
Furthermore, in the case of non-unique solutions (i.e.\ multiple stationary points), it has to be demonstrated that at least one critical point has the form of an APS-deformation. Therefore, we introduce a number of different terms describing the respective solutions.
\begin{itemize}
	\item An \textbf{APS-equilibrium} is an APS-deformation which is stationary with respect to the restriction of the energy functional to the class of APS-deformations (cf.\ Section \ref{sec:constrained-minimization}).
	\item A \textbf{global equilibrium} is an arbitrary deformation which is stationary without the restriction to the class of APS-deformations (cf.\ Section \ref{sec:free-minimization}).
	\item A \textbf{global APS-equilibrium} is an APS-deformation which is a global equilibrium. Note that every global APS-equilibrium is an APS-equilibrium, but that it is not clear a priori whether the converse holds.
	\item We call an energy function $W$ \textbf{APS-admissible} if \emph{every} APS-equilibrium is also a global equilibrium, i.e.\ if every APS-equilibrium is a global APS-equilibrium. Note that APS-admissibility does not imply the existence of an APS-equilibrium. Furthermore, even for a non-APS-admissible energy, it is possible for some specific kind of APS-deformation to be a global APS-equilibrium, cf.\ Remark \ref{remark:SPS}.
\end{itemize}
%
%
%
%
\section{The classical full equilibrium approach}\label{sec:free-minimization}
We first discuss the full equilibrium approach to answer the following closely related questions:\par
\textbf{Under which conditions is every APS-equilibrium a global APS-equilibrium?}\\
\textbf{Under which conditions does a solution of the Euler-Lagrange equations have the form of an APS-deformation?}\par

In general, it is possible to obtain non-APS solutions to the equations of equilibrium which \emph{do} have the form of an APS-deformation on the boundary of $\Omega$, see the example in Figures \ref{fig:displacement}b and \ref{fig:displacement2}b. It is therefore not sufficient to consider the equilibrium equations for APS-deformations exclusively, as described in Section \ref{sec:constrained-minimization}, an approach followed e.g.\ by Gao \cite{gao2017remarks,gao2015duality}.\\
The guiding questions was answered exhaustively by Knowles in 1976 \cite{knowles1976finite,knowles1977note}. In the following, we elaborate on his work and set it in context with the notion of APS-convexity.\par
A \emph{stationary deformation} or \emph{equilibrium solution} $\varphi$ is a deformation which satisfies the Euler-Lagrange equation
\begin{equation}
	\Div[DW(\nabla\varphi)]=0
\end{equation}
to the variational problem
\begin{align}
	I(\varphi)=\displaystyle\int_\Omega W(\nabla\varphi )\,\dx\,\longrightarrow\,\min\,.\label{eq:VariationAPS}
\end{align}
In the following, we will assume that $W(F)\geq W(\id)$, i.e.\ that the natural state $F=\id$ is globally optimal.\par
Any isotropic energy function $W$ can be represented in terms of the invariants $I_1,I_2,I_3$ of $B=FF^T$, see e.g. \citet[Chapter 12.13]{antman}. Since the derivatives of these invariants with respect to $F$ are given by
\begin{align}
	D_FI_1(FF^T).H=\langle 2F,H\rangle,\;\;\, D_FI_2(FF^T).H=\langle 2\,(I_1\id-B)F,H\rangle,\;\;\, D_FI_3(FF^T).H&=\langle 2I_3F^{-T},H\rangle,
\end{align}
respectively, the derivative $DW(F)$ of the energy can be expressed in terms of $I_1$, $I_2$ and $I_3$ via
\begin{align}
	DW(F)&=\frac{\partial W(I_1,I_2,I_3)}{\partial I_1}\.D_FI_1(FF^T)+\frac{\partial W(I_1,I_2,I_3)}{\partial I_2}\.D_FI_2(FF^T)+\frac{\partial W(I_1,I_2,I_3)}{\partial I_3}\.D_FI_3(FF^T)\notag\\
	&=2\.\frac{\partial W}{\partial I_1}F+2\.\frac{\partial W}{\partial I_2}(I_1\id-B)F+2\.I_3\frac{\partial W}{\partial I_3}F^{-T}\,.
\end{align}
In the special case of APS-functions (for which $I_3=1$), the (full) equations of equilibrium are therefore given by (cf.\ \citet[eq.(10)]{knowles1977note})
\begin{align}
	\Div\left(2\.\frac{\partial W}{\partial I_1}F+2\.\frac{\partial W}{\partial I_2}(I_1\id-B)F+p\.F^{-T}\right)=0\label{ELInv}
\end{align}
with $p(I_1,I_2)\colonequals 2\,\frac{\partial W}{\partial I_3}(I_1,I_2,1)$.

Using the general representation \eqref{APSFBC} of APS-functions, eq.\ \eqref{ELInv} can be written more explicitly as
\begin{align}
\text{Div}\left(\begin{matrix}
	2\frac{\partial W}{\partial I_1}+2(2+\beta^2)\frac{\partial W}{\partial I_2}+p & -2\alpha\beta\frac{\partial W}{\partial I_2} & -2\alpha\left(\frac{\partial W}{\partial I_2}+p\right) \\
	-2\alpha\beta\frac{\partial W}{\partial I_2} & 2\frac{\partial W}{\partial I_1}+2(2+\alpha^2)\frac{\partial W}{\partial I_2}+p & -2\beta\left(\frac{\partial W}{\partial I_2}+p\right) \\
	2\alpha\left(\frac{\partial W}{\partial I_1}+\frac{\partial W}{\partial I_2}\right) & 2\beta\left(\frac{\partial W}{\partial I_1}+\frac{\partial W}{\partial I_2}\right) & 2\frac{\partial W}{\partial I_1}+4\frac{\partial W}{\partial I_2}+p
\end{matrix}\right)=0\,.\label{ELAPS}
\end{align}
Since an APS-function is completely defined by a single scalar-valued function $u(x_1,x_2)$, the system \eqref{ELAPS} is \emph{over-determined} by two equations.
In order to ensure the existence of a solution $\overline u(x_1,x_2)$ to all three partial differential equations, the energy function $W$ should therefore satisfy certain conditions so that two equations are omitted.\par
Since all occurring terms in \eqref{ELAPS} are independent of $x_3$, the last column of the matrix does not contribute to the divergence term. We therefore simplify the equations by replacing these entries by $\star$. Introducing the notation
\begin{align}
	G(I_1,I_2)&\colonequals 2\frac{\partial W}{\partial I_2}(I_1,I_2,1)\,,\quad H(I_1,I_2)\colonequals 2\left[\frac{\partial W}{\partial I_1}(I_1,I_2,1)+\frac{\partial W}{\partial I_2}(I_1,I_2,1) \right]\,,\label{notation:GH}\\
	q(I_1,I_2)&\colonequals p(I_1,I_2)+2\.\frac{\partial W}{\partial I_1}(I_1,I_2,1)+2\.(I_1-1)\.\frac{\partial W}{\partial I_2}(I_1,I_2,1) \quad\text{with}\quad p(I_1,I_2)\colonequals 2\,\frac{\partial W}{\partial I_3}(I_1,I_2,1)\,,\notag
\end{align}
we write the Euler-Lagrange equations for the general compressible case as
\begin{align}
	q_{,x_1}&=(\alpha^2 G)_{,x_1}+(\alpha\beta G)_{,x_2}\,,\tag{I}\label{ELAPS_1}\\
	\smash{
		\text{Div}\left(\begin{matrix}
			q-\alpha^2G & -2\alpha\beta G & \star \\
			-2\alpha\beta G & q-\beta^2G & \star \\
			\alpha H & \beta H & \star
		\end{matrix}\right)=0
	}
	\qquad\iff\qquad
	q_{,x_2}&=(\alpha\beta G)_{,x_1}+(\beta^2 G)_{,x_2}\,,\tag{II}\label{ELAPS_2}\\
	\vphantom{\beta^2}
	0&=(\alpha H)_{,x_1}+(\beta H)_{,x_2}\,.\tag{III}\label{ELAPS_3}
\end{align}
Our approach to the problem of over-determination consists of two steps: first, we consider under which circumstances equation \eqref{ELAPS_3} has a solution $\overline u(x_1,x_2)$; next, we find conditions under which this solution $\overline u(x_1,x_2)$ necessarily satisfies the other two equations \eqref{ELAPS_1} and \eqref{ELAPS_2}.
\begin{remark}[Simple plane shear]\label{remark:SPS}
Some classes of APS-deformations satisfy the above equations in a trivial way without further conditions on the energy function. Such deformations are known as \emph{simple plane shear deformations}. These specific APS-deformations solve the full Euler-Lagrange equations \eqref{ELAPS_1}--\eqref{ELAPS_3} without the need of APS-admissibility.\par
The most simple example of a simple plane shear deformation is that of homogeneous shear: Since the Euler-Lagrange equations depend on $\gamma^2=\|\nabla u\|^2$ via $I_1=I_2=3+\gamma^2$, every APS-function with constant $\gamma$ satisfies \eqref{ELAPS_1}--\eqref{ELAPS_3} trivially.\footnote{For the homogeneous deformation $u(x_1,x_2)=c_1\,x_1+c_2\,x_2+c_3$ with constants $c_1,c_2,c_3\in\mathbb{R}\,,$ follows directly from the linearity of $u$ that $\alpha=u_{,x_1}=c_1$ and $\beta=u_{,x_2}=c_2\,.$ This implies $I_1=I_2=3+\alpha^2+\beta^2=\textrm{const.}$, which shows that $G(I_1,I_2)\,,H(I_1,I_2)\,,p(I_1,I_2)\,,q(I_1,I_2)=$ const. Thus, all three Euler-Lagrange equations are trivially fulfilled.}
For a detailed discussion of different types of simple plane shear, including axial-symmetric APS-deformations with $u(x_1,x_2)=\widetilde u(R)$, see \cite{hill1977generalized,agarwal1979finite,carroll2012nonlinear,pucci2015determination}.\par
In the general case, however, APS-boundary conditions do not necessarily allow for simple plane shear deformations. The focus of this work is therefore to elaborate conditions (cf. Section \ref{sec:Conditions}) for APS-admissibility, i.e.\ conditions under which \eqref{ELAPS_1}--\eqref{ELAPS_3} can be satisfied for APS-deformations which are \emph{not} simple plane shear deformations.
\end{remark}
\smallskip
\subsection{APS-convexity}
The third Euler-Lagrange equation, rewritten in divergence form
\begin{align}
	0=\div\left(H\left(3+\norm{\nabla u}^2,3+\norm{\nabla u}^2\right)\nabla u\right)=2\.\div\left(\left(\frac{\partial W}{\partial I_1}(I_1,I_2,I_3)+\frac{\partial W}{\partial I_2}(I_1,I_2,I_3)\right)\nabla u\right)\,,\tag{III}
\end{align}
can be represented as
\begin{align}
	\div\left(g'(\norm{\nabla u}^2)\nabla u\right)=0\,,\quad\text{with}\quad g(x)\colonequals W(3+x,3+x,1)\,.\label{eq:EulerLagrangeAPSConvexity}
\end{align}
From the assumption that the natural state $F=\id$ is globally optimal for $W(F)$, we infer $g(x)\geq g(0)$.
Thus eq.\ \eqref{eq:EulerLagrangeAPSConvexity} is the Euler-Lagrange equation corresponding to the scalar variational problem
\begin{align}
	\int_\Omega\frac{1}{2}\.g(\|\nabla u\|^2)\,\dx&\,\longrightarrow\,\min.\label{eq:VariationalIII}
\end{align}
Of course, the simplest way of ensuring a solution to this equation is to require the convexity of the energy functional.\footnote{Convexity is clearly not necessary for the existence of a minimizer, see e.g. \cite{gao2017remarks}, but it will turn out later that this convexity condition is not a particularly limiting property for most elastic energy functions.} Therefore, the third Euler-Lagrange equation \eqref{ELAPS_3} of our original variation problem does have a solution if the mapping $(\alpha,\beta)\mapsto g(\norm{(\alpha,\beta)}^2)$ is convex.
\begin{definition}[APS-convexity]
We call an energy function $W\col\mathbb{R}^{3\times 3}\,\longrightarrow\,\mathbb{R}\,,\; F\mapsto W(F)$ \emph{anti-plane shear convex} (or \emph{APS-convex}) if it is convex on the convex set
\begin{align*}
	\mathcal{APS}=\bigg\{\left(\begin{matrix}
			1 & 0 & 0\\
			0 & 1 & 0\\
			\alpha & \beta & 1
		\end{matrix}\right) \;\bigg|\; \alpha ,\beta\in\mathbb{R}
	\bigg\}.
\end{align*}
\end{definition}
\begin{remark}
	If the function $W$ is expressed in terms of the matrix invariants $I_1$, $I_2$ and $I_3$, then the function is APS-convex if and only if the mapping $\left(\alpha,\beta\right)\mapsto W(3+\gamma^2,3+\gamma^2,1)$, where $\gamma^2=\alpha^2+\beta^2$, is convex. This equivalence results from the equalities $I_1=I_2=3+\|\nabla u\|^2$ and $I_3=1$ for APS-deformations.
\end{remark}
In the following, we will consider explicit conditions for APS-convexity of isotropic energy functions $W$.
\begin{lemma}\label{lemma:gConvex}
	If $g:[0,\infty)\to\R$ satisfies $g(x)\geq g(0)$ for all $x$ in $\R$, then convexity of $g$ implies APS-convexity of $W(I_1,I_2,I_3)$.
\end{lemma}
\begin{proof}
If $g$ is convex and minimal at $0$, then $g$ is monotone increasing on $[0,\infty)$. Then the mapping $(\alpha,\beta)\mapsto W(3+\gamma^2,3+\gamma^2,1)=g(\norm{(\alpha,\beta)}^2)$ is convex as the composition of the convex mapping $\norm{\,.\,}^2$ and the convex and monotone increasing mapping $g$.
\end{proof}
\begin{lemma}
	If $W(I_1,I_2,I_3)$ is sufficiently smooth and has its global minimum in the natural state then the condition\par
	\noindent\fbox{\parbox{\textwidth}{
		\setlength{\abovedisplayskip}{0pt}
		\setlength{\belowdisplayskip}{0pt}
		\begin{align}
			\forall\,R>0:\qquad\mathcal{W}''(3+R^2)\geq 0\qquad\text{with}\quad\mathcal{W}(I_1)\colonequals W(I_1,I_1,1)\,,\tag{APS1}\label{APS1}
		\end{align}
	}}\par
	implies APS-convexity.
\end{lemma}
\begin{proof}
Condition \eqref{APS1} is equivalent to the convexity of $\mathcal{W}$ on $[3,\infty)$, i.e.\ convexity of the mapping $x\mapsto W(3+x,3+x,1)=g(x)$ on $[0,\infty)$ which, due to Lemma \ref{lemma:gConvex}, implies APS-convexity of $W$.
\end{proof}
The reverse of this implication does not hold in general. In order to obtain a condition equivalent to APS-convexity, we need to directly consider the convexity of the mapping $x\mapsto g(x^2)$ instead.
\begin{theorem}\label{LemmaAPS2}
	The condition\par
	\noindent\fbox{\parbox{\textwidth}{
		\setlength{\abovedisplayskip}{0pt}
		\setlength{\belowdisplayskip}{0pt}
		\begin{align}
			\forall\,R>0:\qquad\frac{d^2}{(dR)^2}\.W(3+R^2,3+R^2,1)\geq 0\,.\tag{APS2}\label{APS4}
		\end{align}
	}}\par
	is equivalent to APS-convexity of $W(F)$.
\end{theorem}
\begin{proof}
	Recall that APS-convexity is equivalent to the convexity of the mapping $(\alpha,\beta)\mapsto g(\norm{(\alpha,\beta)}^2)$, which immediately implies the convexity of the mapping $R\mapsto g(\norm{(R,0)}^2)=g(R^2)=W(3+R^2,3+R^2,1)$ and thus \eqref{APS4}.
	
	If, on the other hand, \eqref{APS4} holds, then the mapping $R\mapsto g(R^2)=W(3+R^2,3+R^2,1)$ is convex and hence, due to the assumed minimality of the energy at the reference configuration, monotone increasing on $[0,\infty)$. Thus the mapping $(\alpha,\beta)\mapsto g(\norm{(\alpha,\beta)}^2)$ is convex as the composition of the (convex) Euclidean norm with a monotone increasing, convex function.
\end{proof}
\begin{remark}
	Under the assumption that $W$ is minimal in $\id$ we have shown that the two statements
	\begin{align*}
		W(I_1,I_2,I_3)\text{ is APS-convex}&:\qquad \left(\alpha\,,\beta\right)\mapsto W(3+\gamma^2,3+\gamma^2,1)\,,\;\;\text{where }\;\gamma^2=\alpha^2+\beta^2\,,\;\;\text{is convex,}\\
		\eqref{APS4}&:\qquad \gamma\mapsto W(3+\gamma^2,3+\gamma^2,1)\text{ is convex on }[0,\infty)
	\end{align*}
	are equivalent.
\end{remark}
\begin{remark}\label{lemma:KnowlesEllipticity}
	The so-called \emph{\enquote{ellipticity condition}}\par
	\noindent\fbox{\parbox{\textwidth}{
		\setlength{\abovedisplayskip}{0pt}
		\setlength{\belowdisplayskip}{0pt}
		\begin{align}
			\forall\,R>0:\qquad\frac{d}{dR}\left[\left. R\left(\frac{\partial W}{\partial I_1}(I_1,I_2,1)+\frac{\partial W}{\partial I_2}(I_1,I_2,1)\right)\right|_{I_1=I_2=3+R^2}\right]\geq 0\,,\tag{APS3}\label{APS3}
		\end{align}
	}}\par
given by \citet[eq.(19)]{knowles1977note} is equivalent to APS-convexity of $W(F)$.
\end{remark}
\begin{proof}
	$\quad\displaystyle\eqref{APS3}\quad\iff\quad\frac{d}{dR}\left[R\.g'(R^2)\right]\geq 0\quad\iff\quad\frac{d^2}{(dR)^2}\.g(R^2)\geq 0\quad\iff\quad\eqref{APS4}$
\end{proof}
The following implication was pointed out by Fosdick et al.\ \cite{fosdick1978transverse,fosdick1973rectilinear}.
\begin{lemma}\label{lemma:fosdick}
	APS-convexity \eqref{APS3} implies
	\begin{align}
			\forall\, R>0:\qquad\mathcal{W}'(3+R^2)&>0\qquad\text{with}\quad\mathcal{W}(I_1)\colonequals W(I_1,I_1,1)\,.
	\end{align}
\end{lemma}
\begin{proof}
	Let $f(R)=R\.\mathcal{W}'(3+R^2)$. Then \eqref{APS3} implies $f'(R)>0$ for all $R>0$, i.e.\ monotonicity of $f$, thus $0=f(0)<f(R)=R\.\mathcal{W}'(3+R^2)$ and hence $\mathcal{W}'(3+R^2)>0$ for all $R>0$.
\end{proof}
\subsection{Energy function admissibility conditions}\label{sec:Conditions}
We introduced APS-convexity as a sufficient condition for the existence of a solution $\overline u(x_1,x_2)$ to equation \eqref{ELAPS_3} and, by means of \eqref{APS4}, derived a simple criterion for this condition. Another way to obtain such a solution $\overline u(x_1,x_2)$ without requiring APS-convexity is discussed in Gao \cite[Theorem 5]{gao2015remarks}, cf.\ Section \ref{sec:constrained-minimization}.

In the following, we consider under which circumstances this solution also satisfies the other two equations \eqref{ELAPS_1} and \eqref{ELAPS_2} so that the APS-deformation induced by $\overline u(x_1,x_2)$ is an overall solution of the full equilibrium equations \eqref{ELAPS_1}--\eqref{ELAPS_3}. The following theorem was obtained by \citet[eq. (21)]{knowles1977note}; here, we want to elaborate on his proof.

Recall that an energy function is APS admissible if every APS-equilibrium (solution of equation \eqref{ELAPS_3}) is also a global equilibrium (solves equations \eqref{ELAPS_1}--\eqref{ELAPS_3}).
\begin{theorem}\label{theoremKnowles}
	{\normalfont (Compressible case)} Let $W(I_1,I_2,I_3)$ be an isotropic, elastic energy function. Then $W$ is APS-admissible if and only if the following conditions are satisfied:\par
	\noindent\fbox{\parbox{\textwidth}{
		\setlength{\abovedisplayskip}{0pt}
		\setlength{\belowdisplayskip}{0pt}
		\begin{align}
			&\exists\, b\in\mathbb{R}:\forall\,I_1=I_2\geq 3\,,I_3=1:&\hspace{-1.5cm} b\,\frac{\partial W}{\partial I_1}(I_1,I_2,I_3)+(b-1)\,\frac{\partial W}{\partial I_2}(I_1,I_2,I_3)&=0\,,\tag{K1}\label{K1}\\
			&\forall\,I_1=I_2\geq 3\,,I_3=1:&\hspace{-1.5cm} \frac{\partial^2W}{\partial {I_1}^2}+I_1\frac{\partial^2W}{\partial I_1\partial I_2}+\frac{\partial^2W}{\partial I_1\partial I_3}+(I_1-1)\frac{\partial^2W}{\partial {I_2}^2}+\frac{\partial^2W}{\partial I_2\partial I_3}+\frac{1}{2}\frac{\partial W}{\partial I_2}&=0\,.\tag{K2}\label{K2}
		\end{align}
	}}\par
\end{theorem}
\begin{proof}
	Recall from Section \ref{sec:free-minimization} that for an APS-deformation, the three equations of equilibrium are given by
	\begin{align}
			q_{,x_1}&=(\alpha^2 G)_{,x_1}+(\alpha\beta G)_{,x_2}\,,&\qquad &\text{(I)}\notag\\
			q_{,x_2}&=(\alpha\beta G)_{,x_1}+(\beta^2 G)_{,x_2}\,,&\qquad &\text{(II)}\\
			0&=(\alpha H)_{,x_1}+(\beta H)_{,x_2}\,.&\qquad &\text{(III)}\vphantom{\beta^2}\notag
		\end{align}
	Let $\overline u(x_1,x_2)$ be an arbitrary solution of equation \eqref{ELAPS_3}, i.e.\ an APS-equilibrium. We want to derive the equations \eqref{K1} and \eqref{K2} as conditions on $W(I_1,I_2,I_3)$ for the other two Euler-Lagrange equations to be necessarily satisfied for $\overline u(x_1,x_2)$.
	
	If relation \eqref{K1} holds,\footnote{For the necessity of \eqref{K1}, see \citet[eq.(3.22)]{knowles1976finite}.} we can simplify equation \eqref{ELAPS_1} to read\footnote{With the notation from \eqref{notation:GH}, we can restate \eqref{K1} as $b\,H(I_1,I_2)=G(I_1,I_2)$ with constant $b$. Therefore, the relationship $\div(H\,\nabla u)=0$ together with $b\,H(I_1,I_2)=G(I_1,I_2)$ yields $\div(G\,\nabla u)=0\,.$}
	\begin{align}
		q_{,x_1}&=(\alpha^2\.G)_{,x_1}+(\alpha\beta\.G)_{,x_2}\;=\;\alpha(\alpha\.G)_{,x_1}+\alpha_{,x_1}\alpha\.G+\alpha(\beta \.G)_{,x_2}+\alpha_{,x_2}\beta\.G\notag\\
		&=  \alpha\,\div(G\,\nabla u)+\alpha_{,x_1}\alpha\.G+\alpha_{,x_2}\beta\.G\;=\; G\.(\alpha_{,x_1}\alpha+\alpha_{,x_2}\beta)\notag\\
		&= G\.(\alpha\alpha_{,x_1}+\beta\beta_{,x_1})\;=\; G\.\frac{\partial}{\partial x_1}\left[\frac{1}{2}\gamma^2\right]=\; G\.\gamma\.\gamma_{,x_1}\,,
	\end{align}
	where $\alpha_{,x_2}=u_{,x_1x_2}=u_{,x_2x_1}=\beta_{,x_1}$.\par
	By utilizing the fact that the invariants\footnote{Note again that $I_1=I_2=3+\gamma^2=3+\|\nabla u\|^2$.} depend only on $u(x_1,x_2)$, the term $q(I_1,I_2)$ can be expressed as
	\begin{align}
		u(x_1,x_2)\mapsto q(3+\gamma^2,3+\gamma^2)\colonequals\widetilde q(\gamma^2)\qquad\text{with}\qquad\frac{\partial q}{\partial x_1}(3+\gamma^2,3+\gamma^2)=\widetilde q\.^\prime(\gamma^2)\.2\.\gamma\.\gamma_{,x_1}\,.
	\end{align}
	Therefore, we can transform \eqref{ELAPS_2} and similarly \eqref{ELAPS_1} to
	\begin{align}
		\left(\widetilde q\.^\prime(\gamma^2)-\frac{1}{2}\,G(3+\gamma^2,3+\gamma^2)\right)2\.\gamma\.\gamma_{,x_1}=0\,,\qquad
		\left(\widetilde q\.^\prime(\gamma^2)-\frac{1}{2}\,G(3+\gamma^2,3+\gamma^2)\right)2\.\gamma\.\gamma_{,x_2}=0\,,
	\end{align}
	respectively. As a result, the Euler-Lagrange equations are simplified by condition \eqref{K1} to the system of equations
	\begin{align}
		\left[\widetilde q\.^\prime(\gamma^2)-\frac{\partial W}{\partial I_2}(3+\gamma^2,3+\gamma^2,1)\right]2\.\gamma\gamma_{,x_1}&=0\,,&\qquad &\text{(I)}\notag\\
		\left[\widetilde q\.^\prime(\gamma^2)-\frac{\partial W}{\partial I_2}(3+\gamma^2,3+\gamma^2,1)\right]2\.\gamma\gamma_{,x_2}&=0\,,&\qquad &\text{(II)}\label{eq:ELgammatoR}\\
		\left[\alpha\,\frac{\partial W}{\partial I_2}(3+\gamma^2,3+\gamma^2,1)\right]_{,x_1}+\left[\beta\,\frac{\partial W}{\partial I_2}(3+\gamma^2,3+\gamma^2,1)\right]_{,x_2} &=0\,.&\qquad &\text{(III)}\notag
	\end{align}
	Note that equations (I) and (II) are trivially satisfied if $\gamma = \|\nabla u\|^2$ is constant, i.e.\ if $\varphi$ is a simple plane shear deformation. In the general case of arbitrary APS-deformations, however, (I) and (II) are satisfied if and only if the equation
	\begin{align}
		\widetilde q\.^\prime(R^2)=\frac{\partial W}{\partial I_2}(3+R^2,3+R^2,1)\,.\label{ELAPS_1+2}
	\end{align}
	holds for all $R\in\R$.
	Thus, equations \eqref{ELAPS_1} and \eqref {ELAPS_2} are reduced to a single new equation \eqref{ELAPS_1+2} by \eqref{K1}. The system of equations is still over-determined by one equation. Therefore, we need to show that the last equation \eqref{ELAPS_1+2} is equivalent to the energy function compatibility condition \eqref{K2}:
	\begin{align*}
		\widetilde q\.^\prime(R^2)=&\ 2\frac{d}{d(R^2)}\frac{\partial W}{\partial I_3}(3+R^2,3+R^2,1)+2\frac{d}{d(R^2)}\frac{\partial W}{\partial I_1}(3+R^2,3+R^2,1)+2\frac{d}{d(R^2)}\left[(2+R^2)\frac{\partial W}{\partial I_2}(3+R^2,3+R^2,1)\right]\\
		=&\ 2\left(\frac{\partial^2W}{\partial I_3\partial I_1}\,1+\frac{\partial^2W}{\partial I_3\partial I_2}\, 1\right)+2\left(\frac{\partial^2W}{\partial {I_1}^2}\, 1+\frac{\partial^2W}{\partial I_1\partial I_2}\, 1\right)+2\.\frac{\partial W}{\partial I_2}+2(2+R^2)\left(\frac{\partial^2W}{\partial I_2\partial I_1}\, 1+\frac{\partial^2W}{\partial {I_2}^2}\, 1\right)\\
		=&\ 2\left[\frac{\partial^2W}{\partial I_1\partial I_3}+\frac{\partial^2W}{\partial I_2\partial I_3}+\frac{\partial^2W}{\partial {I_1}^2}+\frac{\partial^2W}{\partial I_1\partial I_2}+(2+R^2)\left(\frac{\partial^2W}{\partial I_1\partial I_2}+\frac{\partial^2W}{\partial {I_2}^2}\right)+\frac{\partial W}{\partial I_2}\right]\\
		=&\ 2\left[\frac{\partial^2W}{\partial I_1\partial I_3}+\frac{\partial^2W}{\partial I_2\partial I_3}+\frac{\partial^2W}{\partial {I_1}^2}+I_1\frac{\partial^2W}{\partial I_1\partial I_2}+(I_1-1)\frac{\partial^2W}{\partial {I_2}^2}+\frac{\partial W}{\partial I_2}\right]\,.
	\end{align*}
	Thus \eqref{ELAPS_1+2} and \eqref{K2} are, in fact, equivalent in this case.\par
	Altogether, under the two conditions \eqref{K1} and \eqref{K2}, the Euler-Lagrange equations for a compressible energy function are always simplified such that the equations \eqref{ELAPS_1} and \eqref{ELAPS_2} can be omitted for any solution of equation \eqref{ELAPS_3}.
\end{proof}
In the case of incompressible nonlinear elasticity, energy functions are only defined on the special linear group of isochoric deformations with $I_3=1$, thus condition $\eqref{K2}$ is not well defined. However, since APS-deformations belong to the class of isochoric deformations, the problem of APS-admissibility can be considered in the incompressible case as well.
It should be expected that in the incompressible case, less restricting requirements than the conditions $\eqref{K1}$ and $\eqref{K2}$ are needed to ensure APS-admissibility.\par 
The concept of APS-convexity remains the same for incompressible and compressible energy functions, starting with the variational problem
\begin{align}
     	\displaystyle\int_\Omega W(\nabla\varphi )\,\dx\,\longrightarrow\underset{\det \nabla\varphi =1}{\min}\quad\implies\quad \displaystyle\int_\Omega W(\nabla\varphi )+p(x)\,(\det(\nabla\varphi) -1)\,\dx\,\longrightarrow\,\min.,
\end{align}
where $p(x_1,x_2,x_3)\in C^1(\Omega)$ is now the Lagrange multiplier for the constraint $\det\nabla\varphi=1$ of incompressiblity. With the same notation as before, the Euler-Lagrange equations are simplified to
\[
	\Div\left[DW(F)+p(x)\.\Cof(F)\right]=0
\]
with $\Cof(F)=(\det F)\.F^{-T}= F^{-T}$ by incompressibility. We obtain the same formal equation as in the compressible case \eqref{ELInv}:
\begin{align}
	\Div\left(2\frac{\partial W}{\partial I_1}F+2\frac{\partial W}{\partial I_2}(I_1\id-B)F+pF^{-T}\right)=0\,.\label{ELInvInc}
\end{align}
Here, however, $p\in C^1(\Omega,\mathbb{R})$ is the Lagrange multiplier and not a fixed term given by the energy function $W(F)$. This yields the same equilibrium system of three coupled partial differential equations, but this time in two scalar-valued functions $u(x_1,x_2)$ and $p(x_1,x_2,x_3)$. Therefore, the equilibrium system is only over-determined by one equation, which means that although the system still does not have a solution in general, only one condition on the energy function is required for APS-admissibility.
\begin{theorem}
	{\normalfont (Incompressible case)} Let $W(I_1,I_2)$ be an isotropic and \textbf{incompressible} elastic energy function. The function $W$ is APS-admissible if and only if\par
	\noindent\fbox{\parbox{\textwidth}{
		\setlength{\abovedisplayskip}{0pt}
		\setlength{\belowdisplayskip}{0pt}
		\begin{align}
			\exists\. b\in\mathbb{R}:\forall\,I_1=I_2\geq 3:\qquad b\.\frac{\partial W}{\partial I_1}(I_1,I_2)+(b-1)\.\frac{\partial W}{\partial I_2}(I_1,I_2)=0\,.\tag{K1}
		\end{align}
	}}\par
\end{theorem}
\begin{proof}
	Analogously to the proof of Theorem \eqref{theoremKnowles}, the Euler-Lagrange equations can be reduced with the condition \eqref{K1} by one equation. Therefore, we can remove one of the first two Euler-Lagrange equations and leave two equations to determine $u(x_1,x_2)$ and $p(x_1,x_2,x_3)$. The system of equations is therefore no longer over-determined under the assumption of \eqref{K1}. Moreover, it is possible to compute the Lagrange-multiplier $p(x_1,x_2,x_3)$ for a given solution $\overline u(x_1,x_2)$.\footnote{For detailed calculations, see \cite{voss2017master}.}
\end{proof}
\begin{remark}
	For $\frac{\partial W}{\partial I_1}=c_1$ and $\frac{\partial W}{\partial I_2}=c_2$ with arbitrary constants $c_1,c_2>0$, condition \eqref{K1} is automatically satisfied with $b=\frac{c_2}{c_1+c_2}$ and the energy function is APS-convex \eqref{APS3}.
\end{remark}
\begin{remark}
In linear elasticity, the energy function $W_{\rm{lin}}(\varepsilon)=\mu\,\|\varepsilon\|^2+\frac{\lambda}{2}\,\text{tr}(\varepsilon)^2$ with $\varepsilon=\sym\nabla u$ is automatically APS-admissible and APS-convex \cite{voss2017master}. Therefore, any linear elasticity solution constrained by APS-boundary conditions is automatically an APS-deformation. Thus APS-admissibility is an inherently nonlinear concept.
\end{remark}
%
%
%
%
\section{Connections to constitutive requirements in nonlinear\\ elasticity}
The concept of APS-convexity can be extended to the class of APS$^\textbf{+}$-deformations $\varphi :\Omega\subset\mathbb{R}^3\to\mathbb{R}^3$,
\begin{align}
	\varphi(x_1,x_2, x_3)= \left(x_1,\;x_2,\;x_3+u(x_1, x_2, x_3)\right)\quad\text{with}\quad\varphi\in C^1(\Omega)\,.\label{eq:scalarType}
\end{align}
We call convexity of this type of functions \textbf{APS}$^\textbf{+}$\textbf{-convexity}. Note that APS$^+$-convexity immediately implies APS-convexity.
\subsection{Convexity}\label{sec:convex}
The following lemma shows that an energy function $W$ is APS$^+$-convex (and thus APS-convex) if it is \emph{polyconvex}, i.e.\ if \cite[eq.(0.8)]{ball1976convexity}
\begin{align*}
	&W(F)=P(F,\:\text{Cof}(F),\:\text{det}(F))\quad\text{with}\quad P:\mathbb{R}^{3\times 3}\times\mathbb{R}^{3\times 3}\times\mathbb{R}\cong\mathbb{R}^{19}\,\longrightarrow\,\mathbb{R}\quad\text{convex}\,.
\end{align*}
\begin{lemma}\label{LemmaPolyAps}
	Every polyconvex energy function $W(F)$ is $\text{APS$^+$}$-convex.
\end{lemma}
\begin{proof}
	For $\text{APS}^+$-convexity of $W$ in $F=\nabla\varphi$ we have to show that
	\begin{align}
		\quad W(t\:\nabla\varphi_1+(1-t)\:\nabla\varphi_2)\leq t\:W(\nabla\varphi_1)+(1-t)\:W(\nabla\varphi_2)\,,\quad t\in [0,1]\notag
	\end{align}
	holds for arbitrary  APS$^\textbf{+}$-deformations $\varphi_1, \varphi_2$ \eqref{eq:scalarType}. In this case, the minors of $F=\nabla\varphi$ are given by
	\begin{align}
		F=\left(\begin{matrix}
			1 & 0 & 0 \\
			0 & 1 & 0 \\
			u_{,x_1} & u_{,x_2} & 1+u_{,x_3}
		\end{matrix}\right),\quad\text{ Cof}(F)=\left(\begin{matrix}
			1+u_{,x_3} & 0 & -u_{,x_1} \\
			0 & 1+u_{x_3} & -u_{,x_2} \\
			0& 0 & 1
		\end{matrix}\right),\quad\text{ det}(F)=1+u_{,x_3}\,.\label{skaF}
	\end{align}
	Due to the affine linearity of the above terms, we find for $\varphi =t\,\varphi_1+(1-t)\, \varphi_2$:
	\begin{align}
		F&=t\,F_1+(1-t)\,F_2\,,\notag\\
      	\Cof(t\,F_1+(1-t)\,F_2)&=t\,\text{Cof}(F_1)+(1-t)\,\text{Cof}(F_2)\,,\\
      	\det(t\,F_1+(1-t)\,F_2)&=t\,\text{det}(F_1)+(1-t)\,\text{det}(F_2)\,,\notag
	\end{align}
	where $F=\nabla\varphi\,.$ If $P$ is convex, then
	\begin{align*}
		W(t\,F_1+(1-t)\,F_2)
		&=P(t\,F_1+(1-t)\,F_2,\,\text{Cof}(t\,F_1+(1-t)\,F_2),\,\text{det}(t\,F_1+(1-t)\,F_2))\\
		&=P(t\,F_1+(1-t)\,F_2,\,\text{Cof}(F_1)+(1-t)\,\text{Cof}(F_2),\,t\,\text{det}(F_1)+(1-t)\,\text{det}(F_2))\\
		&\leq t\,P(F_1,\,\text{Cof}(F_1),\,\text{det}(F_1))+(1-t)\,P(F_2,\,\text{Cof}(F_2),\,\text{det}(F_2))\\
		&=t\,W(F_1)+(1-t)\,W(F_2)\,.\qedhere
	\end{align*}
\end{proof}
We now want to reduce the requirement of polyconvexity to that of rank-one convexity. An energy function $W(F)$ is called \emph{rank-one convex} if the mapping $t\mapsto W(F+t\,\xi\otimes\eta)$ is convex on $[0,1]$ for all $F\in\mathbb{R}^{3\times 3}$ and all $\xi,\eta\in\mathbb{R}^3$.
\begin{lemma}\label{lemma:Rank1APS}
	Every rank-one convex energy function $W(F)$ is $\text{APS}^+$-convex.
\end{lemma}
\begin{proof}
	Again, we need to show that the mapping
	\[
		t\mapsto W(t\.F_1+(1-t)\.F_2) = W(F_2+t\.(F_1-F_2))
	\]
	is convex on $[0,1]$ for all $F_1,F_2$ of the form \eqref{skaF}$_1$. However, this convexity property follows directly from the rank-one convexity since $F_1-F_2$ is of the form
	\[
		F_1-F_2 =
			\left(\begin{matrix}
				0 & 0 & 0 \\
				0 & 0 & 0 \\
				u_{,x_1}-v_{,x_1} & u_{,x_2}-v_{,x_2} & u_{,x_3}-v_{,x_3}
			\end{matrix}\right)=\left(\begin{array}{c}
	  			0 \\0\\1
	  		\end{array}\right)\otimes\left(\begin{array}{c}
			  	u_{,x_1}-v_{,x_1} \\ u_{,x_2}-v_{,x_2} \\ u_{,x_3}-v_{,x_3}
	  		\end{array}\right)
	  	.\qedhere
	\]
\end{proof}
\begin{remark}
	The above proof also shows that $W$ is APS$^+$-convex if and only if the mapping $t\mapsto W(F+t\,(0,0,1)^T\otimes\eta)$ is convex on $[0,1]$ for all $F$ of the form \eqref{skaF}$_1$ and all $\eta\in\mathbb{R}^3$. 
\end{remark}
\begin{corollary}
	If $W(F)$ is strictly rank-one convex and APS-admissible, then the anti-plane shear solution (APS-equilibrium) is a unique APS-equilibrium and minimal in the class of APS-deformations, due to APS-convexity.
\end{corollary}
\begin{remark}
	As demonstrated by Lemma \ref{lemma:Rank1APS}, APS-convexity is not a highly restrictive condition for physically viable elastic energy functions. Moreover, it is remarkable that APS-convexity is equivalent to the monotonicity of the Cauchy shear stress in simple shear, see Appendix \ref{lemma:CauchyMonotonicity}.
\end{remark}
\begin{remark}
\label{remark:pucciCounterexample}
	In a recent article by Pucci et al.\ \cite[eq.(7.1)]{pucci2015determination} it is claimed that in the compressible case, Knowles' \quoteref{ellipticity condition \quoteesc{\ldots} is a consequence of the empirical inequalities and \quoteesc{compatibility with linear elasticity}}, i.e.\ that the so-called \emph{empirical inequalities} \cite{moon1974interpretation,antman,truesdell1952}
	\begin{align}
		\beta_0&\colonequals\frac{2}{\sqrt{I_3}}\left(I_2\.\frac{\partial W}{\partial I_2}+I_3\.\frac{\partial W}{\partial I_3}\right)\leq 0\,,\qquad\beta_1\colonequals\frac{2}{\sqrt{I_3}}\.\frac{\partial W}{\partial I_1}>0\,,\qquad\beta_{-1}\colonequals-2\sqrt{I_3}\.\frac{\partial W}{\partial I_2}\leq 0\label{eq:empiricalInequalities}
	\end{align}
	 together with the condition of a stress-free reference configuration imply Knowles' ellipticity condition (Remark \ref{lemma:KnowlesEllipticity}), cf.\ Appendix \ref{remark:Pucci}. We show here that for large deformations, this statement is erroneous: Consider the energy function
	\begin{align}
		W(F)=\frac{3\.\mu}{4}\.\alpha\.\big[\,\underbrace{\log(I_1)+\log(I_2)-\log(I_3)}_{=\log(\norm{U}^2)+\log(\norm{U^{-1}}^2)}-2\log(3)\.\big]+\frac{\mu}{2}\.(1-\alpha)\left[I_1+\frac{2}{\sqrt{I_3}}-5\right]\label{eq:PucciEnergy}
	\end{align}
	with $\mu>0$ and $0<\alpha<1$. The first term is isochoric and therefore has bulk modulus $\kappa=0$, the second term ensures positive bulk modulus in the reference state. The empirical inequalities
	\begin{align}
		\beta_0&=\frac{3\.\mu}{4}\.\alpha \left[\frac{2}{\sqrt{I_3}}\left(I_2\.\frac{1}{I_2}-I_3\.\frac{1}{I_3}\right)\right]+\frac{\mu}{2}\.(1-\alpha)\left[\frac{2}{\sqrt{I_3}}\left(0-I_3\.\frac{1}{I_3^{3/2}}\right)\right]= \frac{\mu\.(1-\alpha)}{I_3}<0\,,\\
		\beta_1&=\frac{3\.\mu}{4}\.\alpha\.\frac{2}{\sqrt{I_3}}\.\frac{1}{I_1}+\frac{\mu}{2}\.(1-\alpha)\.\frac{2}{\sqrt{I_3}}\cdot 1>0\,,\qquad\beta_{-1}=-\frac{3\.\mu}{4}\.\alpha\.\sqrt{I_3}\.\frac{2}{I_2}\leq 0
	\end{align}	 
	 are satisfied. Moreover, the energy function is stress-free in the reference configuration $F=\id$, since
	 \begin{align}
	 	\left[\frac{\partial W}{\partial I_1}+2\.\frac{\partial W}{\partial I_2}+\frac{\partial W}{\partial I_3}\right]_{F=\id}=\frac{3\.\mu}{4}\.\alpha\left[\frac{1}{3}+\frac{2}{3}-\frac{1}{1}\right]+\frac{\mu}{2}\.(1-\alpha)\left[1-\frac{1}{1}\right]=0
	 \end{align}
	 and the generated infinitesimal shear modulus can be determined from
	 \begin{align}
	 	\left(\beta_1-\beta_{-1}\right)_{F=\id}=\frac{3\.\mu}{4}\.\alpha\left[\frac{2}{3}+\frac{2}{3}\right]+\frac{\mu}{2}\.(1-\alpha)\left[\frac{2}{1}\right]=\mu\,.
	 \end{align}
	 Recall from Lemma \ref{lemma:KnowlesEllipticity} that Knowles' ellipticity condition is equivalent to the condition \eqref{APS4} of APS-convexity which, in this case, reads
	 \begin{align}
	 	0 &\leq \frac{d^2}{(dR)^2}\.W(3+R^2,3+R^2,1)=\frac{3\.\mu}{4}\.\alpha\.\frac{d^2}{(dR)^2}\left[2\.\log(3+R^2)\right]+\frac{\mu}{2}\.(1-\alpha)\frac{d^2}{(dR)^2}\left[R^2\right]\\
	 	&=\frac{3\.\mu}{2}\.\alpha\.\frac{d}{dR}\left[\frac{2\.R}{3+R^2}\right]+\frac{\mu}{2}\.(1-\alpha)\frac{d}{dR}\left[2R\right]=3\.\mu\.\alpha\left[\frac{3-R^2}{\left(3+R^2\right)^2}\right]+\mu\.(1-\alpha)\,.\notag
	 \end{align}
	  However, if $\frac{8}{9}<\alpha<1$, then there exists an interval where APS-convexity is violated. Therefore, the energy function \eqref{eq:PucciEnergy} with $\frac{8}{9}<\alpha$ is compatible with linear elasticity and satisfies the empirical inequalities \eqref{eq:empiricalInequalities} as well as the condition of a stress-free reference configuration, but does not satisfy Knowles' ellipticity condition, in contradiction to the claim by Pucci et al.\ \cite{pucci2015determination}.
\end{remark}
\subsection{Tension-compression symmetry}
Table \ref{tab:1} shows a number of elastic energy potentials used in nonlinear elasticity theory and their properties regarding APS-convexity. The detailed calculations can be found in \cite{voss2017master}.\par
\hspace{-1.4cm}
\renewcommand{\arraystretch}{1.6}
\begin{tabularx}{1.15\textwidth}{m{0.13\textwidth}|m{0.56\textwidth}|X|X|c|X}
	Name & Energy expression & Rank1-convex & APS-convex & K1 incomp. & K2\par compress.\\
\hline
	vol.+iso. Neo-Hooke\textsuperscript{\cite{Ogden83}}	& $W(F)=\frac{\mu}{2}(I_1I_3^{-\frac{1}{3}}-3)+h(I_3)$		& Yes & Yes & $b=0$  & No\\
\hline
	vol.+iso. Mooney-Rivlin\textsuperscript{\cite{Ogden83}}	& $W(F)=\frac{\mu}{2}\left(\alpha(I_1I_3^{-\frac{1}{3}}-3)+(1-\alpha)(I_2I_3^{-\frac{2}{3}}-3)\right)+h(I_3)$		& Yes & Yes & \mbox{$b=1-\alpha$}  & No\\
\hline
	Blatz-Ko\textsuperscript{\cite{horgan1996remarks}}	& $W(F)=\frac{\mu}{2}\left(I_1+\frac{2}{\sqrt{I_3}}-5\right)$		& Yes & Yes	& $b=0$  & Yes\footnote{A general class of APS-admissible energy functions $W(I_1,I_3)$ can be found in \cite{jiang2001exact}.}\\
\hline
	Veronda-Westman\textsuperscript{\cite{oberai2009linear}}	& $W(F)=\mu\left(\frac{e^{\gamma(I_1-3)}-1}{\gamma}-\frac{I_2-3}{2}\right)+h(I_3)$	& No & Yes & No  & No\\
\hline
	Mihai-Neff\textsuperscript{\cite{mihai2017hyperelastic,neff2017injectivity}}	& $W(F)=\frac{\mu}{2}\left(I_1\,{I_3}^{-\frac{1}{3}}-3\right)+\frac{\widetilde{\mu}}{4}\,(I_1-3)^2+\frac{\kappa}{2}\left({I_3}^\frac{1}{2}-1\right)^2$		& No & Yes	& $b=0$  & \mbox{$\widetilde{\mu}=\frac{\mu}{3}$}\\[1ex]
\hline
	Knowles	& $W(F)=\frac{\mu}{2b}\left(\left[1+\frac{b}{n}\left(I_1\,{I_3}^{-\frac{1}{3}}-3\right)\right]^n-1\right)+\frac{1}{D_1}\left({I_3}^\frac{1}{2}-1\right)^2$	& ? & Yes	& $b=0$  & No\\
\hline
	Bazant	& $W(F)=\|B-B^{-1}\|^2$		& No & Yes & $b=\frac{1}{2}$ & No\\
\hline
	Ciarlet\textsuperscript{\cite{Ciarlet1988}}	& $W(F)=\frac{c_1}{2}\.\norm{F}^2+\frac{c_2}{2}\.\norm{\Cof F}^2+h(\det F)$\newline\hspace*{1.03cm}$=\frac{c_1}{2}\.I_1+\frac{c_2}{2}\.I_2+h(\sqrt{I_3})$	& Yes & Yes	& \mbox{$b=\frac{c_2}{c_1+c_2}$}  & \mbox{$c_2=0$}\\[1ex]
\hline
	SVK\textsuperscript{\cite{Ciarlet1988}}	& $W(F)=\frac{\mu}{4}\,\|C-\id\|^2+\frac{\lambda}{8}\,\text{tr}(C-\id)^2$		& No & Yes & No & —\\
\hline
	4th Order	& $W(F)=\mu\,\tr(E^2)+\frac{A}{2}\,\tr(E^3)+D\,\tr(E^2)^2$		& No & Yes & No & —\\
\hline
	Hencky\textsuperscript{\cite{Hencky1929,neff2016geometry}}	& $W(F)=\mu\|\dev\log V\|^2+\frac{\kappa}{2}\left(\tr(\log V)\right)^2$	& No & No & $b=\frac{1}{2}$ & No\\
\hline
	\mbox{exp-Hencky\textsuperscript{\cite{neff2015exponentiated}}}	& $W(F)=\frac{\mu}{k}e^{k\,\|\dev\log V\|^2}+\frac{\kappa}{2\hat{k}}e^{\hat{k}\left(\tr(\log V)\right)^2}$		& No & Yes	& $b=\frac{1}{2}$  & No\\
\hline
	Martin-Neff	& $W(F)=\frac{\|F\|^3}{\det(F)}+\det(F)\,\|F^{-1}\|^3$		& Yes & Yes & $b=\frac{1}{2}$ & No\\
\hline
Model\textsuperscript{\cite{voss2017master}}	& $W(F)=c_1\left(\sqrt{I_1}+\sqrt{I_2}+\frac{\sqrt{3}}{\sqrt{I_3}}-3\sqrt{3}\right)$\par	No viable approximation to linear elasticity in $F=\id\,.$	& Yes & Yes & $b=\frac{1}{2}$ & Yes
\end{tabularx}
\captionof{table}{An overview of APS-related properties for several important energy functions.}\label{tab:1}
\vspace{0.2cm}
\renewcommand{\arraystretch}{0.625}

Note that an APS-admissible energy in the incompressible case only has to satisfy condition (K1), whereas an APS-admissible energy for the general compressible case must also fulfill condition (K2). A still unsolved problem is to find a compressible viable energy function which is APS-admissible but depends nonlinearly on $I_2$. It is noticeable in Table \ref{tab:1} that many energy functions satisfy condition \eqref{K1} with $b=0$ or $b=\frac{1}{2}$; the former case can be easily explained by the independence from the second invariant.

\begin{lemma}
	Every isotropic energy function $W(F)$ which can be expressed in the form $W(F)=W(I_1,I_3)$, i.e.\ which does not depend on the second invariant $I_2$, satisfies condition \eqref{K1} with $b=0$.
\end{lemma}
\begin{proof}
Condition \eqref{K1} with $b=0$ can be simplified to $\frac{\partial W}{\partial I_2}=0$, which is trivially fulfilled for every isotropic energy function of the type $W(F)=W(I_1,I_3)$.
\end{proof}
The special case $b=\frac{1}{2}$, on the other hand, shows a more interesting relation to the so-called \emph{tension-compression symmetry} of an energy.
\begin{definition}
	An energy function $W(F)$ is called \emph{tension-compression symmetric} if $W(F)=W(F^{-1})$ for all $F\in\GL^+(3)$.
\end{definition}
\begin{lemma}\label{lemma:tensionCompressionSymmetric}
	An isotropic tension-compression symmetric energy function $W$ is invariant under permutation of the two invariants $I_1$ and $I_2$ under the constraint of incompressiblity, i.e.\ $W(I_1,I_2,1)=W(I_2,I_1,1)$.
\end{lemma}
\begin{proof}
	Let $I_1'=I_1(B^{-1})\,,\quad I_2'=I_2(B^{-1})\,,\quad I_3'=I_3(B^{-1})\,.$ Then
	\begin{align}
		I_1'&=\tr(B^{-1})=\tr\Big(\frac{\det(B)}{\det(B)}\,B^{-1}\Big)=\frac{1}{\det(B)}\tr(\det(B)\,B^{-T})=\frac{\tr(\Cof(B))}{\det(B)}=\frac{I_2}{I_3}\,,\\[1ex]
		I_2'&=\tr(\Cof(B^{-1}))=\tr(\det(B^{-1})\,(B^{-1})^{-T})=\det(B^{-1})\,\tr(B^T)=\frac{\tr(B)}{\det(B)}=\frac{I_1}{I_3}\,,\\[1ex]
		I_3'&=\det(B^{-1})=\frac{1}{\det(B)}=\frac{1}{I_3}\,.
	\end{align}
	Therefore, tension-compression-symmetry implies $W(I_1,I_2,I_3)=W(I_1',I_2',I_3')=W(\frac{I_2}{I_3},\frac{I_1}{I_3},\frac{1}{I_3})$ and thus, in particular, $W(I_1,I_2,1)=W\left(\frac{I_2}{1},\frac{I_1}{1},\frac{1}{1}\right)=W(I_2,I_1,1)$.
\end{proof}
\begin{theorem}
	Every isotropic tension-compression-symmetric energy function $W(F)$ satisfies condition \eqref{K1} with $b=\frac{1}{2}$.
\end{theorem}
\begin{proof}
	The condition \eqref{K1} with $b=\frac{1}{2}$ can be restated as
	\[
		\frac{1}{2}\,\frac{\partial W}{\partial I_1}(I_1,I_1,1)+\Big(\frac{1}{2}-1\Big)\frac{\partial W}{\partial I_2}(I_1,I_1,1)=0
		\quad\iff\quad
		\frac{\partial W}{\partial I_1}(I_1,I_1,1)=\frac{\partial W}{\partial I_2}(I_1,I_1,1)
	\]
	for all $I_1\geq3$, and for tension-compression symmetric $W$ we find
	\[
		\frac{\partial W}{\partial I_1}(I_1,I_1,1)=\left.\frac{d}{dt}W(t,I_1,1)\right|_{t=I_1} = \left.\frac{d}{dt}W(I_1,t,1)\right|_{t=I_1}=\frac{\partial W}{\partial I_2}(I_1,I_1,1)
	\]
	due to Lemma \ref{lemma:tensionCompressionSymmetric}.
\end{proof}
Coming back to Table \ref{tab:1}, we observe that no energy function which satisfies condition \eqref{K1} with $b=\frac{1}{2}$ also fulfills the second condition \eqref{K2}. Therefore, we hypothesize that APS-admissibility is not a reasonable characteristic for physically motivated compressible energy functions.
%
%
%
%
\section{The constrained equilibrium approach}\label{sec:constrained-minimization}
By testing several examples, we are led to believe that most viable energy functions in compressible nonlinear elasticity are \emph{not} APS-admissible. Therefore, in general, APS-boundary conditions do \emph{not} necessarily lead to an APS-deformation of the whole body. Nevertheless, it is possible to compute the energetically optimal APS-deformation by minimization only over the class of APS-functions:
\begin{align}
	I(\varphi)=\int_\Omega W(\nabla\varphi)\,\dx\,\longrightarrow\,\underset{\varphi \in\mathcal{APS}}{\min}\,.\label{AnsatzRW}
\end{align}
An equilibrium of the corresponding Euler-Lagrange equations of \eqref{AnsatzRW} (with respect to the restriction of the energy functional to the class of APS-deformations) is called APS-equilibrium and does not have to be stationary in the global sense \eqref{eq:VariationAPS}. As emphasized by Saccomandi \cite{saccomandi2016dy}, this approach was chosen by Gao \cite{gao2015remarks,gao2017remarks,gao2015duality,gao2016analytical}\footnote{\citet{gao2014}:\enquote{ [\ldots] the equilibrium equation [\ldots] has just one non-trivial component [namely equation (III)].} Gao claims that Knowles' condition \eqref{K1} is automatically satisfied for every elastic energy function with $b=0$, which is clearly not the case (Table \ref{tab:1}).}: starting with
\begin{align}
	I(u)=\displaystyle\int_\Omega \mathcal{W}(3+\|\nabla u\|^2)\,\dx\,\longrightarrow\,\min\,,
\end{align}
where we employ the same notation\footnote{For APS-deformations, $I_1=I_2=3+\|\nabla u\|^2$ and $I_3=1$.} $W(I_1,I_1,1)=\mathcal{W}(I_1)$ as before, we obtain the Euler-Lagrange equation $\div(\mathcal{W}'(3+\|\nabla u\|^2)\.\nabla u)=0$ for stationarity within the class of APS-deformations, which is equivalent to equation \eqref{ELAPS_3} from the full equilibrium approach. 
\begin{corollary}
	APS-Convexity of the energy function $W(F)$ ensures the existence of a unique APS-equilibrium which is a global energy minimizer (among the class of APS-deformations).
\end{corollary}
Contrary to Theorem 5 in \cite{gao2017remarks}, we see in Lemma \eqref{lemma:Rank1APS} that strict rank-one convexity implies strict APS-convexity which, in turn, implies uniqueness of the APS-equilibrium.\par
Gao \cite{gao2017remarks} prominently discusses the case where $g(\norm{u}^2)=\mathcal{W}(3+\norm{u}^2)=W(3+\norm{u}^2,3+\norm{u}^2,1)$ is \emph{not} convex. In this case, the existence of a solution to the minimization problem is not clear due to the loss of APS-convexity (see Lemma \ref{lemma:gConvex}), and one needs to resort to just solutions of the Euler-Lagrange equation \eqref{ELAPS_3}; of course, while such solutions may exist, it is by no means obvious why they should satisfy the general equations of equilibrium.
\begin{remark}
	If an energy function is APS-admissible (satisfies \eqref{K1} for incompressible material behavior or \eqref{K1} and \eqref{K2} in the compressible case), then the full and the constrained equilibrium approach provide the same solution.
\end{remark}
%
%
%
%
\section{Finite element simulations}
We consider the deformation of a unit cube $\Omega$ with APS-type Dirichlet boundary conditions on the four lateral sides of the cube, see Figure \ref{fig:BoundaryConditions}.
\begin{figure}[H]
	\hspace{0.12\textwidth}
	\includegraphics[clip, trim=1.5cm 0cm 9cm 0cm,width=0.35\textwidth]{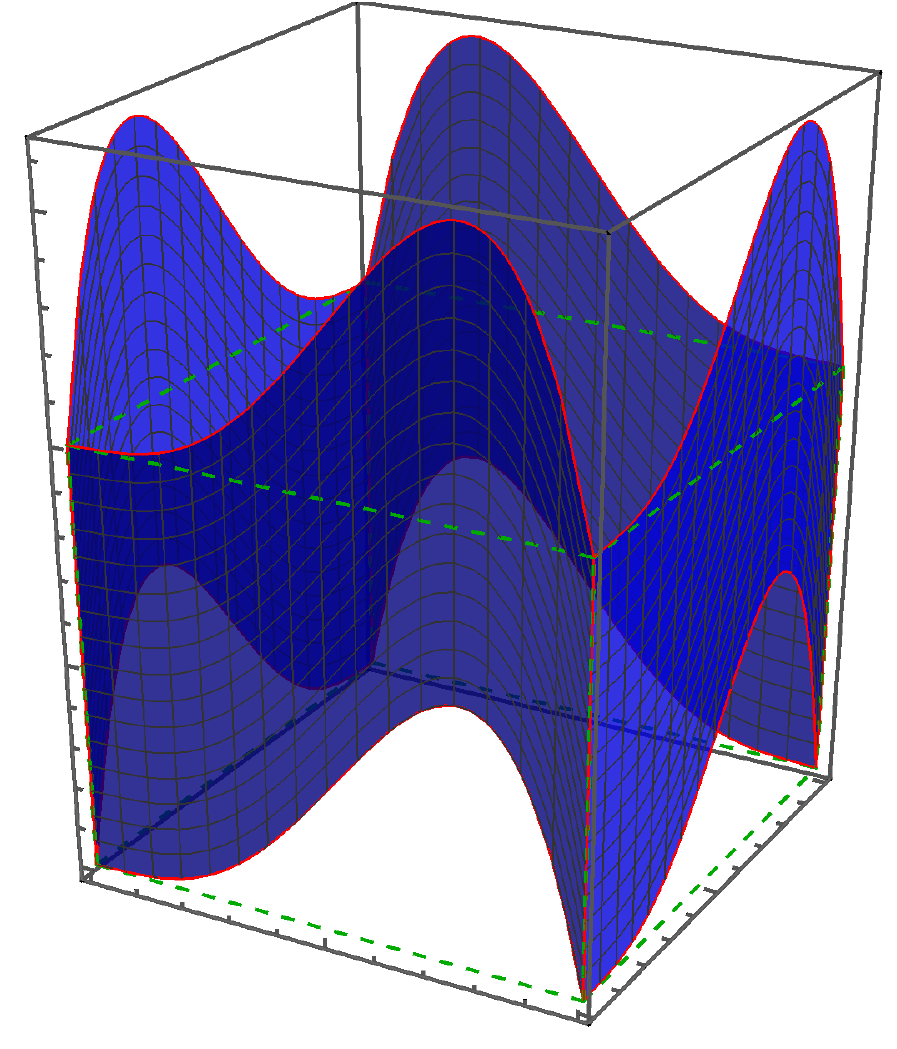}
	\hspace{0.08\textwidth}
	\includegraphics[clip, trim=10.5cm 0cm 0cm 0cm,width=0.35\textwidth]{Randwerte.pdf}
	\caption{\label{fig:BoundaryConditions}(Left) Prescribed APS-boundary conditions. (Right) possible APS-deformation of the top and bottom side of the cube.}
\end{figure}
In order to compare the APS-computations for different constitutive laws, we perform numerical simulations using the finite element system \textsc{Abaqus} \cite{Aba17}, which supports the use of internal models (e.g.\ the compressible Neo-Hooke or the compressible Mooney-Rivlin model) as well as the implementation of custom hyperelastic models via the provided user subroutine \texttt{uhyper}, which requires the user to provide the energy function $W(I_1,I_2,I_3)$ in terms of the invariants as well as its first, second and third derivatives.
For our numerical calculations, we use a grid of $21\times 21\times 21$ nodes. The considered unit cube is discretized by 8-noded linear brick elements with hybrid formulation (C3D8H) in order to get better approximations for the (quasi-)incompressible hyperelastic models.

The APS-boundary conditions as shown in Figure \ref{fig:BoundaryConditions} are realized by the \texttt{disp} subroutine which enables the user to prescribe values for selected node sets and their addressed degree of freedom (DOF) for each iteration increment. Here, we apply the functional value depending on the nodal $x_1,x_2$-position onto the boundary nodes of the unit cube.\par
In the following, we want to visualize the difference between APS-admissible energy functions in the general compressible case, APS-admissibility only for the constraint of incompressibility and an energy function that satisfies neither condition. We start with the incompressible case and choose the Mooney-Rivlin and Veronda-Westman energy functions (see Table \ref{tab:1}). Both are APS-convex, but only the Mooney-Rivlin energy satisfies the condition \eqref{K1} which implies APS-admissibility in the incompressible case.\\
An exact APS-deformation is characterized by an exclusive displacement in $e_3$-direction for every node of the whole body $\Omega$. Therefore, the $e_1$-$e_2$-plane grid-structure of the nodes in the undeformed body $\Omega$ has to be maintained by any deformation in equilibrium for an APS-admissible energy function. We introduce the measure $u_\delta=\sqrt{(\varphi_1(x)-x_1)^2+(\varphi_1(x)-x_2)^2}$ of deviation from an APS-deformation.
\begin{figure}[H]
	\includegraphics[clip, trim=1.7cm 0cm 0cm 0cm,width=0.5\textwidth]{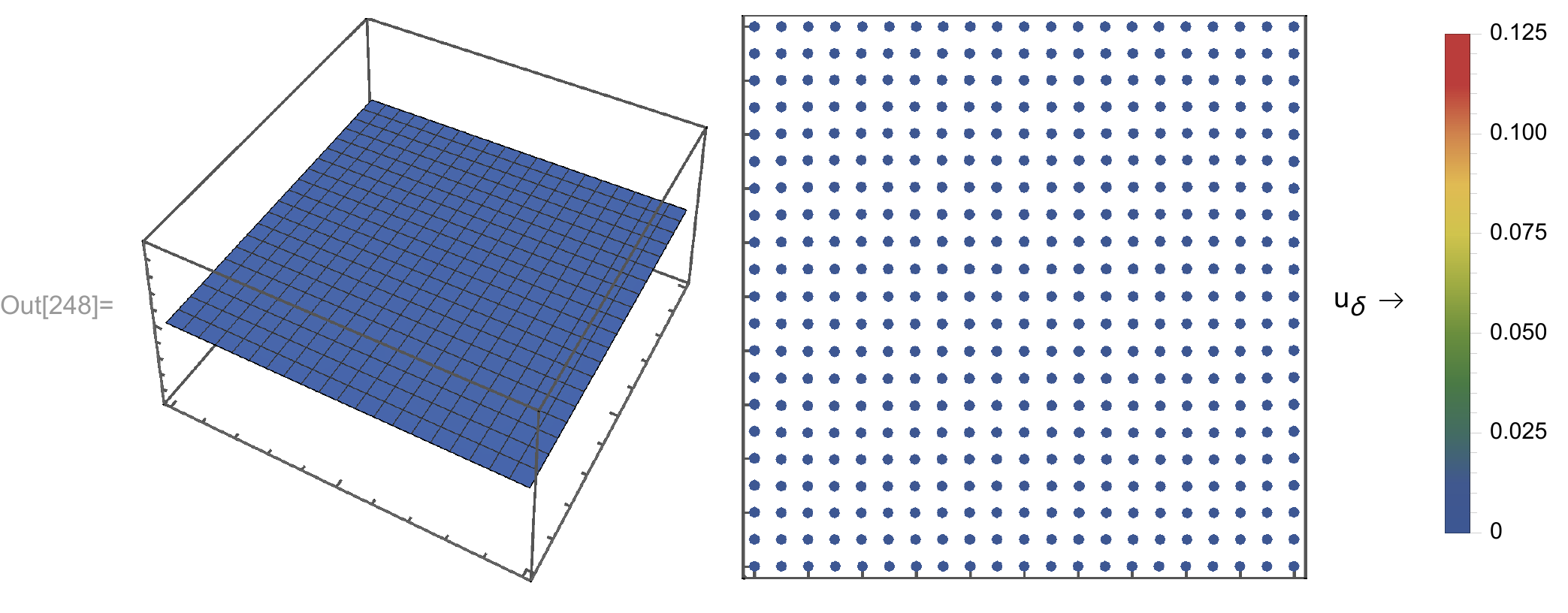}
	\hspace{-8.3cm}
	a)
	\hspace{7.6cm}
	\includegraphics[clip, trim=1.6cm 0cm 0cm 0cm,width=0.5\textwidth]{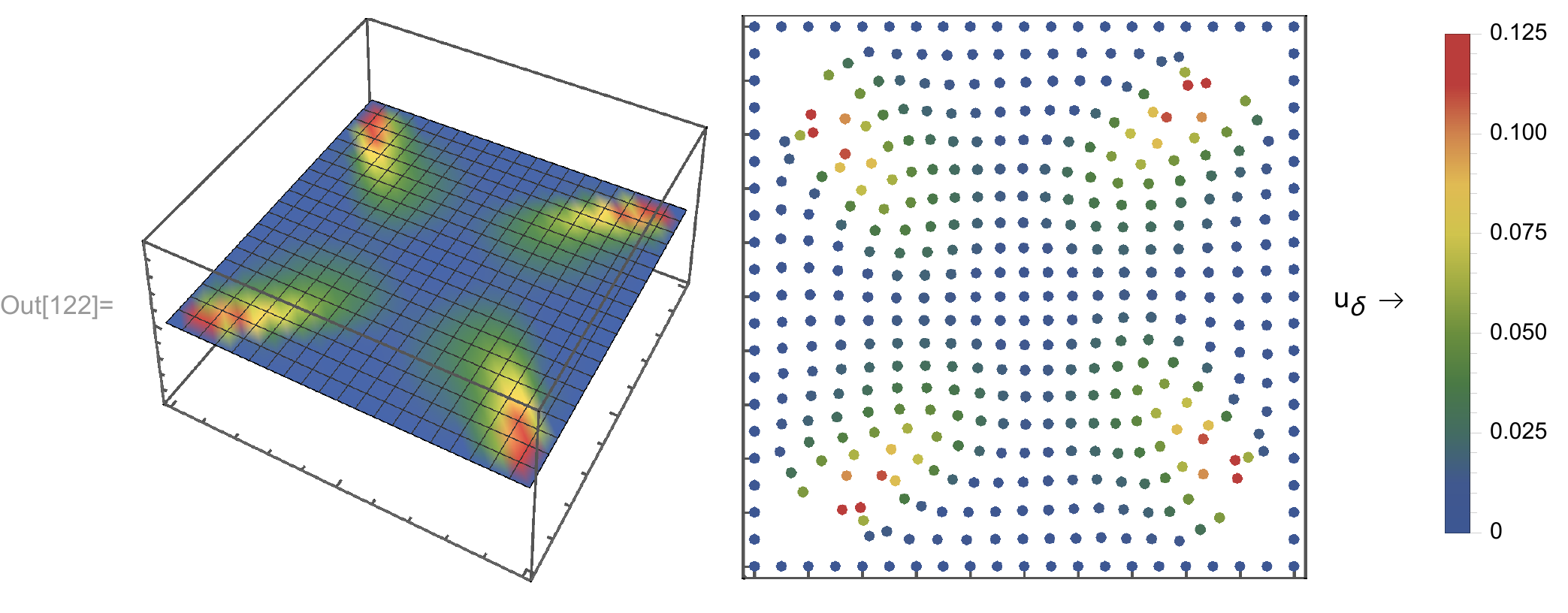}
	\hspace{-8.3cm}
	b)
	\caption{\label{fig:displacement}Visualization of a slice (at $x_3=0.5$) of the deformed square with APS-boundary condition (Figure \ref{fig:BoundaryConditions}) for the Mooney-Rivlin (a) and the Veronda-Westman (b) energy in the quasi-incompressible case (bulk modulus $K\sim 10^5\mu$ shear modulus). The color shows the displacement $u_\delta$ in $x_1$- and $x_2$-direction.}
\end{figure}
The graphics in Figure \ref{fig:displacement}a visualize the deformation induced by the Mooney-Rivlin energy, which is APS-admissible in the incompressible case. The slice of the inside of the cube shows perfect APS-behavior, maintaining the original grid structure. The deformation induced by the non APS-admissible Veronda-Westman energy function is shown in Figure \ref{fig:displacement}b. The deviation to the original grid-structure is more distinct and affects the whole body $\Omega$.\\
For the visualization of APS-admissibility in the compressible case, we again use the Mooney-Rivlin energy and compare it to the APS-admissible Blatz-Ko model (cf.\ Table \ref{tab:1}); note that the Mooney-Rivlin energy is not APS-admissible in the compressible case.
\begin{figure}[H]
	\includegraphics[clip, trim=1.6cm 0cm 0cm 0cm,width=0.5\textwidth]{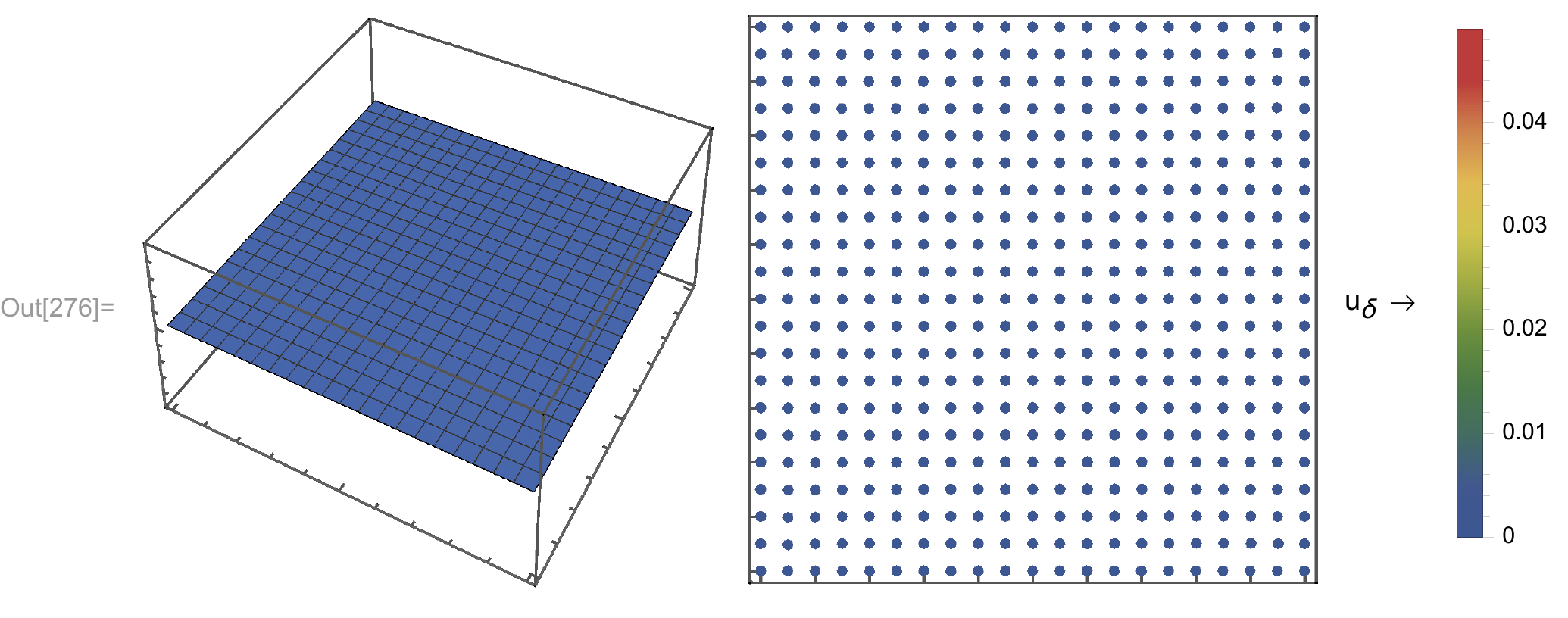}
	\hspace{-8.3cm}
	a)
	\hspace{7.6cm}
	\includegraphics[clip, trim=1.5cm 0cm 0cm 0cm,width=0.5\textwidth]{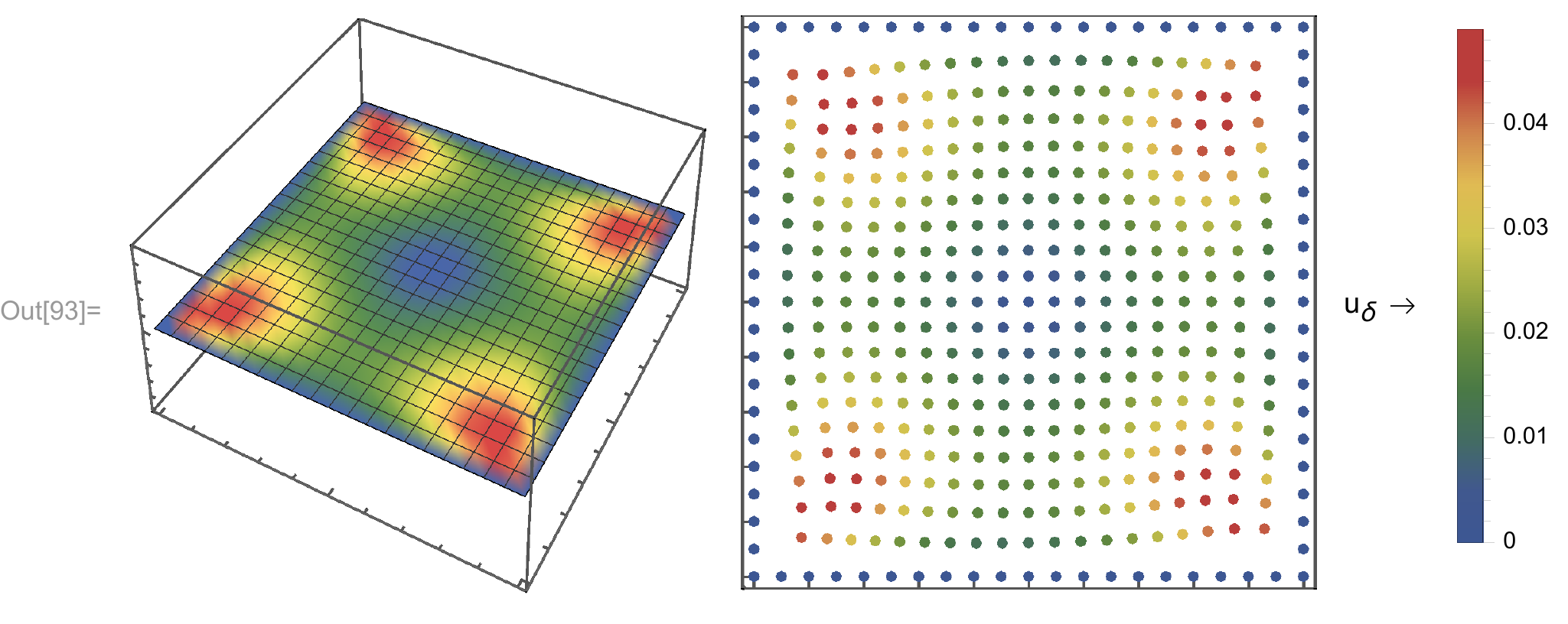}
	\hspace{-8.3cm}
	b)
	\caption{\label{fig:displacement2}Visualization of a slice ($x_3=0.5$) of the deformed square with APS-boundary condition (Figure \ref{fig:BoundaryConditions}) with the Blatz-Ko (a) and the Mooney-Rivlin (b) energy in the compressible case (bulk modulus $K\sim\mu$ shear modulus). The color shows the displacement $u_\delta$ in $x_1$- and $x_2$-direction.}
\end{figure}
Similar to our observation of the quasi-incompressible case, the equilibrium solution for the APS-admissible energy function (Blatz-Ko, Figure \ref{fig:displacement2}a) shows perfect APS-behavior inside the cube. The deformation induced by the Mooney-Rivlin energy, on the other hand, shows more distinguished deviations from an APS-deformation throughout the whole body.
%
%
%
%
\section{Conclusion}
This work elaborates on Knowles' paper \cite{knowles1977note} and the difference between Knowles' (full) and Gao's (constrained) approach. The two conditions \eqref{K1} and \eqref{K2} as discovered by Knowles were derived directly from the Euler-Lagrange-equations of the energy function $W$. The required ellipticity condition in \citet[eq.(19)]{knowles1977note} was identified with the introduced APS-convexity and inferred from to the important concepts of polyconvexity and rank-one convexity. Since the latter condition is a highly desirable property in nonlinear elasticity from a mathematical point of view, the requirement of APS-convexity does not further restrict the class of viable energy functions. Moreover, investigating different elastic energy functions revealed that even a number of commonly used non-rank-one convex energy functions are in fact still APS-convex due to its equivalence to the monotonicity of the Cauchy shear stress in simple shear.\par
Furthermore, it was shown that, contrary to expectations, \eqref{K1} is fulfilled by almost all investigated energy functions; indeed, this condition follows from the physically reasonable requirement of tension-compression-symmetry. Therefore, it is to be expected for incompressible elastic materials to exhibit APS-deformations under given APS-boundary conditions.\par
However, within the context of nonlinear hyperelasticity, the additional APS-admissibility condition \eqref{K2} in the compressible case appears to be satisfied only in the trivial case of energies which do not depend on on the second invariant $I_2$. Note that the only energy function listed in Table \ref{tab:1} that satisfies condition \eqref{K2} and is not independent of $I_2$ is a function without consistency to linear elasticity, unsuitable for mechanical application.\par
By numerical simulations, we were able to visualize the difference between the deformations under APS-type boundary conditions induced by APS-admissible and non-APS-admissible energy functions in the incompressible as well as the compressible case.
%
%
%
\section*{Acknowledgement}
We thank Giuseppe Saccomandi (University of Perugia) and Roger Fosdick (University of Minnesota) for helpful discussions.

%
%
%
%
\footnotesize
\section{References}
\printbibliography[heading=none]
\begin{appendix}
\section{Appendix}
Recall that in the isotropic case, the Cauchy-stress tensor can always be expressed in the form
\begin{align}
	\sigma=\beta_0\.\id+\beta_1\.B+\beta_{-1}\.B^{-1}
\end{align}	
with scalar-valued functions $\beta_i$ depending on the invariants of $B$. In the hyperelastic isotropic case, $\beta_0$, $\beta_1$ and $\beta_{-1}$ are given by
\begin{align}
	\beta_0=\frac{2}{\sqrt{I_3}}\left(I_2\.\frac{\partial W}{\partial I_2}+I_3\.\frac{\partial W}{\partial I_3}\right)\,,\qquad\beta_1=\frac{2}{\sqrt{I_3}}\.\frac{\partial W}{\partial I_1}\,,\qquad\beta_{-1}=-2\sqrt{I_3}\.\frac{\partial W}{\partial I_2}\,.\label{eq:CauchyHyper}
\end{align}

\begin{lemma}\label{lemma:CauchyMonotonicity}
	Let $\varphi\col\Omega\to\R$, $\varphi(x)=(x_1+\gamma\.x_2,x_2,x_3)$ be a simple shear deformation, with $\gamma\in\R$ denoting the amount of shear. Then the Cauchy shear stress $\sigma_{12}$ of an arbitrary isotropic energy function $W(I_1,I_2,I_3)$ is monotone as a scalar-valued function depending on the amount of shear for positive $\gamma$ \emph{if and only if} $W$ is APS-convex.
\end{lemma}
\begin{proof}
	We consider the Cauchy-stress tensor for an arbitrary material which is stress-free in the reference configuration:
	\begin{align}
		\sigma=\beta_0\.\id+\beta_1\.B+\beta_{-1}\.B^{-1}\quad\text{with}\quad \beta_0&=\frac{2}{\sqrt{I_3}}\left(I_2\.\frac{\partial W}{\partial I_2}+I_3\.\frac{\partial W}{\partial I_3}\right)\,,\qquad\beta_1=\frac{2}{\sqrt{I_3}}\.\frac{\partial W}{\partial I_1}\,,\qquad\beta_{-1}=-2\sqrt{I_3}\.\frac{\partial W}{\partial I_2}\,.
	\end{align}
	In the case of simple shear we compute \cite[p.41]{beatty2001seven}
	\begin{align}
		\nabla\varphi&=\left(\begin{matrix}1&\gamma&0\\0&1&0\\0&0&1\end{matrix}\right)\,,\qquad B=FF^T=\left(\begin{matrix}1+\gamma^2&\gamma&0\\\gamma&1&0\\0&0&1\end{matrix}	\right)\,,\qquad B^{-1}=\left(\begin{matrix}1&-\gamma&0\\-\gamma&1+\gamma^2&0\\0&0&1\end{matrix}\right)\,,\\
		I_1&=\tr B=3+\gamma^2\,,\qquad I_2=\tr(\Cof B)=\tr\left(\begin{matrix}1&-\gamma&0\\-\gamma&1+\gamma^2&0\\0&0&1\end{matrix}\right)=3+\gamma^2\,,\qquad I_3=\det B=1\,,\\
		\implies\qquad\sigma&=(\beta_0+\beta_1+\beta_{-1})\.\id+\left(\begin{matrix}\beta_1\.\gamma^2&(\beta_1-\beta_{-1})\.\gamma&0\\(\beta_1-\beta_{-1})\.\gamma&\beta_{-1}\.\gamma^2&0\\0&0&0\end{matrix}\right)\,.
	\end{align}
	Therefore, the Cauchy shear stress component $\sigma_{12}$ is a scalar-valued function depending on the amount of shear $\gamma$, given by
	\begin{align}
		\sigma_{12}(\gamma)&=(\beta_1-\beta_{-1})\.\gamma=\gamma\left.\frac{2}{\sqrt{I_3}}\left(\frac{\partial W}{\partial I_1}+I_3\.\frac{\partial W}{\partial I_2}\right)\right|_{I_1=I_2=3+\gamma^2,I_3=1}\\
		&=2\gamma\left.\left(\frac{\partial W}{\partial I_1}+\frac{\partial W}{\partial I_2}\right)\right|_{I_1=I_2=3+\gamma^2,I_3=1}=\frac{d}{d\gamma}\.W(3+\gamma^2,3+\gamma^2,1)\,.
	\end{align}
	The positivity of the Cauchy shear stress is already implied by the (weak) empirical inequalities $\beta_1>0\,,\;\beta_{-1}\leq 0$. The condition for shear-monotonicity is given by
	\begin{align}
		\frac{d}{d\gamma}\.\sigma_{12}(\gamma)=\frac{d^2}{(d\gamma)^2}\.W(3+\gamma^2,3+\gamma^2,1)>0\qquad\qquad\forall\,\gamma\geq
		 0\,,
	\end{align}
	which is equivalent to APS-convexity condition \eqref{APS4} of the energy function $W(I_1,I_2,I_3)\,.$
\end{proof}
\begin{remark}\label{remark:Pucci}
	The empirical inequalities \eqref{eq:empiricalInequalities} state that $\beta_0\leq 0\,,\;\beta_1>0\,,\;\beta_{-1}\leq 0$. In the case of APS-deformations ($I_1=I_2=3+\gamma^2\,,I_3=1$), Pucci et al.\ \cite[eq.(4.3)]{pucci2015determination} obtain the inequality
	\begin{align}
		(I_1-3)^p(h^*)'+q^2h^*>0\,,\qquad\forall\,I_1\geq3\qquad\text{with }\; h^*(I_1)=\left.\beta_1-\beta_{-1}\right|_{I_1=I_2\,,I_3=1}\,,\label{eq:Pucci4.3}
	\end{align}
	\quoteref{where $p,q$ are real numbers such that $p>0$ and $q\neq0$}, by a \quoteref{simple manipulation of the empirical inequalities \quoteesc{\eqref{eq:empiricalInequalities}} and \quoteesc{the stress free reference configuration}}. In \cite[Remark III]{pucci2015determination}, it is pointed out correctly that in the case of $p=1\,,q^2=\frac{1}{2}$ (they erroneously use $q=1$) the resulting constitutive inequality
	\begin{align}
		0<2\left((I_1-3)\.(h^*)'(3+\gamma^2)+\frac{1}{2}\.h^*(3+\gamma^2)\right)=(h^*)'(3+\gamma^2)\cdot 2\.\gamma^2+h^*(3+\gamma^2)=\frac{d}{d\gamma}\left[\gamma\.h^*(3+\gamma^2)\right]
	\end{align}
	is equivalent to APS-convexity by equation \eqref{APS3} with
	\begin{align}
		\left.\left(\frac{\partial W}{\partial I_1}+\frac{\partial W}{\partial I_2}\right)\right|_{I_1=I_2=3+\gamma^2,I_3=1}\overset{\eqref{eq:CauchyHyper}}{=}\left.\beta_1-\beta_{-1}\right|_{I_1=I_2=3+\gamma^2\,,I_3=1}=h^*(3+\gamma^2)\,.
	\end{align}
	We are, however, not able to reproduce a proof of inequality \eqref{eq:Pucci4.3}, see also the counterexample in Remark \ref{remark:pucciCounterexample}.
\end{remark}
\begin{lemma}\label{lemma:NeighborhoodShearMonotonicity}
	Let $W$ be a sufficiently smooth isotropic energy function such that the induced Cauchy stress response satisfies the (weak) empirical inequalities.
	Then for sufficiently small shear deformations (i.e.\ within a neighborhood of the identity $\id$), the Cauchy shear stress is a \emph{monotone} function of the amount of shear.
\end{lemma}
\begin{proof}
	In Lemma \ref{lemma:CauchyMonotonicity}, we already computed the Cauchy shear stress corresponding to a simple shear to be $\sigma_{12}(\gamma)=(\beta_1-\beta_{-1})\.\gamma$, with $\gamma\in\R$ denoting the amount of shear.	The monotonicity of this mapping is equivalent to
	\begin{align}
		0<\frac{d}{d\gamma}\sigma_{12}(\gamma)=\frac{d}{d\gamma}\left[(\beta_1(3+\gamma^2)-\beta_{-1}(3+\gamma^2))\.\gamma\right]=\left(\beta_1'(3+\gamma^2)-\beta_{-1}'(3+\gamma^2)\right)2\.\gamma^2+\beta_1(3+\gamma^2)-\beta_{-1}(3+\gamma^2)\,.
	\end{align}
	According to the (weak) empirical inequalities, $\beta_1(3)-\beta_{-1}(3)\equalscolon\mu>0$. Therefore, $\beta_1(3+\gamma^2)-\beta_{-1}(3+\gamma^2)\geq\eps>0$ for sufficiently small $\gamma\in\R$. If $W$ and thus $\beta_1,\beta_{-1}$ are sufficiently smooth, then $\beta_1'-\beta_2'$ is locally Lipschitz-continuous, and thus within a compact neighborhood of $\id$,
	\[
		\frac{d}{d\gamma}\sigma_{12}(\gamma) = \underbrace{\beta_1(3+\gamma^2)-\beta_{-1}(3+\gamma^2)}_{\geq\eps} \,+\, \underbrace{\left(\beta_1'(3+\gamma^2)-\beta_{-1}'(3+\gamma^2)\right)}_{\leq\;\text{const.}} \,\cdot\,2\.\gamma^2>0
	\]
	for every sufficiently small shear deformation, i.e.\ sufficiently small $\gamma$.
\end{proof}
\end{appendix}
\end{document}